\title{Concyclic Centroids and the Anti-Fuhrmann Triangle}

\documentclass[a4paper,11pt,leqno]{amsart}
\usepackage{amsfonts}
\usepackage{amssymb}
\usepackage{amsmath}
\usepackage{graphicx}
\usepackage{wrapfig}
\usepackage{hyperref}

\newcommand{\HH}{\mathcal H}

\newcommand{\tri}{\triangle}
\newcommand{\ma}{\measuredangle}

\newtheorem{theorem}{Theorem}[section]
\newtheorem{corollary}{Corollary}[section]

\newtheorem{lemma}{Lemma}[section]
\newtheorem{prop}{Proposition}[section]
\newtheorem{definition}{Definition}[section]

\newtheorem{rem}{Remark}[section]

\begin{document}
	
	\begin{abstract}
		In this article we present a remarkable concyclicity of four centroids related to the orthocenters of the triangles $ABC,$ $BPC,$ $CPA,$ and $PAB$ of a quadrangle $ABCP$. Furthermore, we establish a result about orthologic triangles inscribed in a rectangular hyperbola. We present an application of the concyclic centroids theorem when $P$ corresponds to the incenter $I$ of $ABC$. This specialization leads to the discovery of a novel triangle intimately related to the Fuhrmann triangle.
	\end{abstract}
	\title{}
	
	\markboth{Sudharshan K V, Shantanu Nene}{Concyclic Centroids and the Anti-Fuhrmann Triangle}
	
	\centerline {\Large{\bf CONCYCLIC CENTROIDS AND THE ANTI-FUHRMANN TRIANGLE}}
	
	\bigskip
	
	\begin{center}
		{\large SUDHARSHAN K V, SHANTANU NENE}
		
		\centerline{}
	\end{center}
	
	\bigskip
	
	\textbf{Abstract.} In this article we present a remarkable concyclicity of four centroids related to the orthocenters of the triangles $ABC,$ $BPC,$ $CPA,$ and $PAB$ of a quadrangle $ABCP$. Furthermore, we establish a result about orthologic triangles inscribed in a rectangular hyperbola. We present an application of the concyclic centroids theorem when $P$ corresponds to the incenter $I$ of $ABC$. This specialization leads to the discovery of a novel triangle intimately connected to the Fuhrmann triangle.
	
	\bigskip
	
	\section{Introduction}
	Let $ABC$ be a triangle, and let $P$ be a point, and let $H,H_A,H_B,$ and $H_C$ be the orthocenters of triangles $ABC, BPC,CPA,$ and $PAB$ respectively. Let $\HH_P$ be a circumrectangular hyperbola of $ABC$ passing through $P$. Let $Z$ denote the Poncelet point of $ABCP$, which is also the center of $\HH_P$. The following two lemmas are well-known.
	
	\begin{lemma}
		\label{lemma:basic}
		A circumconic $\HH$ of $ABC$ is a rectangular hyperbola if and only if $\HH$ contains $H$.
	\end{lemma}
	\begin{proof}
		Suppose $\HH$ is a rectangular hyperbola. Then, $\HH$ passes through two points at infinity along perpendicular directions, which we denote by $\infty_1, \infty_2$.  We have an equality of cross ratios $$(BA, BH; B\infty_1, B\infty_2) = (CH, CA; C\infty_2, C\infty_1),$$ since each line in the second pencil is perpendicular to the corresponding line in the first pencil. It immediately follows that $(BA, BH; B\infty_1, B\infty_2)$ $= (CA, CH; C\infty_1, C\infty_2)$. The point $H$ thus lies in the conic $\HH$. \\
		
		Conversely, suppose that $\HH$ contains $H$. Let $\infty_1$ be a point at infinity on $\HH$, and let $\infty_2$ be a point at infinity in the direction perpendicular to $\infty_1$. Once again, from an equality of cross ratios $(BA, BH; B\infty_1, B\infty_2) =$  $(CH, CA; C\infty_2, C\infty_1)$, we obtain an equality $$(BA, BH; B\infty_1, B\infty_2)=(CA, CH; C\infty_1, C\infty_2),$$ which implies that $\infty_2$ lies on $\HH$.  
	\end{proof}
	
	\begin{corollary}
		\label{cor:ortho}
		Let $A,B,C$ be three points on a rectangular hyperbola $\HH$. If the perpendicular from $A$ to $BC$ intersects $\HH$ again at $D$, then $D$ is the orthocenter of $ABC$.
	\end{corollary}
	
	\bigskip
	
	\bigskip
	
	\bigskip
	
	--------------------------------------
	
	\textbf{Keywords and phrases: }Poncelet Points, Fuhrmann Triangle, Nagel Point
	
	\textbf{(2020)Mathematics Subject Classification: }51M04, 51P99 
	
	\newpage
	
	\begin{lemma}
		\label{lemma:basic2}
		Let $\HH_P$ intersect the circumcircle of $ABC$ at $H'$. Then, $H,H'$ are antipodal points in $\HH_P$.
	\end{lemma}
	
	\begin{proof}
		Because of Lemma \ref{lemma:basic}, $\HH_P$ is the unique rectangular hyperbola containing $H'$. Hence, $Z$ is also the Poncelet point of $ABCH'$. The Poncelet point of a cyclic quadrilateral is its anticenter (See \cite{ivan}, \textbf{Theorem 3.4}). But the anticenter of $ABCH'$ is the midpoint of $HH'$.  
	\end{proof}
	
	\section{Preliminaries}
	Let $\tri_A, \tri_B$, and $\tri_C$ denote the triangles $AH_BH_C, BH_CH_A$, and $CH_AH_B$ respectively.
	\begin{prop}
		\label{prop:common_ortho}
		The triangles $\tri_A, \tri_B, \tri_C$ have a common orthocenter.
	\end{prop}
	\begin{proof}
		See that $AH_C$ and $CH_A$ are parallel since they both are perpendicular to $BP$. Let $X$ be the orthocenter of $\tri_A$. Then, $XH_B \perp AH_C$, so $XH_B \perp CH_A$. From Lemma \ref{lemma:basic}, $X$ lies on $\HH_P$. Applying Corollary \ref{cor:ortho}, we get that $X$ is the orthocenter of $\tri_C$. A similar argument  shows that $X$ is also the orthocenter of $\tri_B$.
	\end{proof}
	\begin{rem}
		The usage of Cartesian coordinates provides straightforward proofs to the above and many of the results in this section.
	\end{rem}
	We retain the notation from the proof, and denote the common orthocenter by $X$. 
	\begin{corollary}
		\label{cor:conc}
		The circumcircles of $\tri_A, \tri_B,$ and $\tri_C$ are concurrent at a point on $\HH_P$.
	\end{corollary}
	
	\begin{figure}[!h]
		\includegraphics{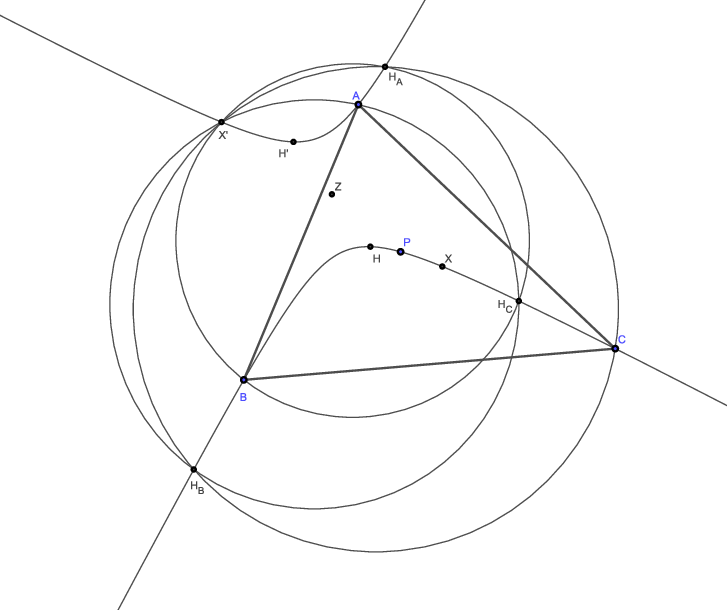}
		\caption{Concurrent Circumcircles}
		\label{fig:conc_circum}
	\end{figure}
	
	We can further characterize the common point of these circumcircles.
	\begin{prop}[Navneel Singhal]
		\label{prop:navneel}
		Let $Q$ be the isogonal conjugate of $P$, and let $Q^*$ denote the inverse of $Q$ in the circumcircle of $ABC$. Then the circumcircles of $\tri_A, \tri_B,$ and $\tri_C$ pass through the midpoint of $QQ^*$.
	\end{prop}
	Let $P'$ denote the antipode of $P$ in $\HH_P$. We require a Lemma by AoPS user \textit{enhanced} for the proof of the proposition.
	\begin{lemma}
		Let $X'$ be the antipode of $X$ in $\HH_P$. Then $HX' \parallel PP'$.
	\end{lemma}
	\begin{proof}
		Let $\tau$ denote the tangent line to $\HH_P$ at $P$. Pascal's Theorem applied to $PXAHPH_A$ shows that the points $HX \cap \tau,$ $AP \cap XH_A,$ and $AH \cap PH_A$ are collinear. Now, $AP$ and $XH_A$ are both perpendicular to $BH_C$, so their intersection is at infinity. Similarly, $AH$ and $PH_A$ are perpendicular to $BC$ and so their intersection is also at infinity. Hence, $HX$ is parallel to $\tau$. \\
		The line joining the midpoints of parallel chords of a conic passes through the center. So, $PZ$ bisects the segment $HX$. But, $HXH'X'$ is a parallelogram with center $Z$. So, $PZ$ $\parallel HX'$.
	\end{proof}
	\\
	
	\begin{proof}
		(Of Proposition \ref{prop:navneel}) The circumconic $\HH_P$ is the isogonal conjugate of the line $OQ$. Choose points $Q_1, \dots, Q_4$ on the line $OQ$, and let $P_1, \dots, P_4$ denote their isogonal conjugates. Then we have an equaliy of cross ratios $(AQ_1, AQ_2; AQ_3, AQ_4)$ $= (AP_1, AP_2; AP_3, AP_4)$ from which we obtain an equality 
		\begin{equation}
			(Q_1, Q_2; Q_3, Q_4)_{OQ} = (P_1, P_2; P_3, P_4)_{\HH_P}, \label{eq:1}
		\end{equation} 
		where the first and second cross ratios are taken with respect to line $OQ$ and conic $\HH_P$ respectively.   \\
		Let $Q_\infty, M$ denote the point at infinity along $OQ$ and the midpoint of $QQ^*$ respectively. The pairs $(P', Q^*)$ and $(H', Q_\infty)$ are isogonal conjugates (see \cite{ivan}, \textbf{Theorem 2.2}). Let $M^*$ denote the isogonal conjugate of $M$. Then from (\ref{eq:1}), we obtain $(P, P'; M^*, H')_{\HH_P} = (Q,Q^*; M, Q_\infty)_{OQ} = -1$. \\
		
		If $P_\infty$ is the point at infinity along $PP'$, the Lemma gives $$(P,P'; X', H')_{\HH_P}\stackrel{H}{=} (P, P', P_\infty, Z) = -1,$$ where the first equality is obtained by projecting from $H$. So, $X'$ is the isogonal conjugate of $M$. But $X'$ is common to the circumcircles of $\tri_A, \tri_B,$ and $\tri_C$ from Lemma \ref{lemma:basic2}.
	\end{proof}
	\subsection{On Orthologic Triangles inscribed in a Rectangular Hyperbola}
	We present a result about orthologic triangles which follows immediately from our discussions above. 
	\begin{prop}
		\label{prop:orthologic}
		Let $A,B,C,D,E,F$ be six points on a rectangular hyperbola $\HH$. Suppose that triangles $ABC$ and $DEF$ are orthologic, and suppose further that one of the orthology centers, $P$, lies on $\HH$. Then the other orthology center lies on $\HH$ as well.
	\end{prop}
	\begin{proof}
		Without losing generality, assume that $DP \perp BC$, $EP \perp AC$, and $FP \perp AB$. Then, Corollary \ref{cor:ortho} says that $D,E,F$ are the orthocenters of triangles $BPC,CPA,APB$ respectively. Let $Q$ be the common orthocenter of triangles $AEF, BFD, CDE$ (see Proposition \ref{prop:common_ortho}). Then, $Q$ lies on $\HH$ and $AQ \perp EF, BQ\perp FD, CQ\perp DE$, so $Q$ is the other orthology center of $ABC$ and $DEF$.
	\end{proof}
	
	\section{Concyclic Centroids}
	Let $G_A,G_B,G_C,$ and $G$ denote the centroids of $\tri_A, \tri_B, \tri_C,$ and $ABC$ respectively. We have a remarkable concyclicity of the points $G,G_A,G_B,$ and $G_C$, and we state it as a Theorem. 
	\begin{figure}[h]
		\includegraphics{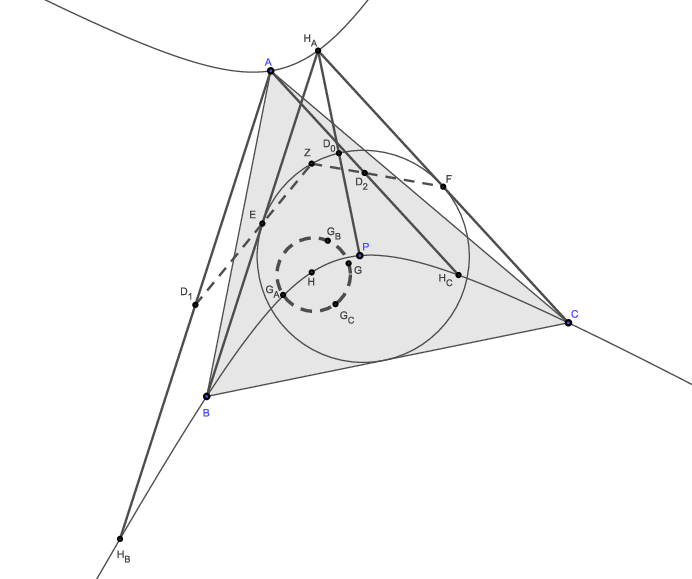}
		\caption{Concyclic Centroids}
	\end{figure}
	\begin{theorem}[Concyclic Centroids]
		\label{thm:concyclic}
		$G_A,G_B,G_C,$ and $G$ lie on a circle.
	\end{theorem}
	
	\begin{proof}
		First, we claim that $GG_A \perp AP$. Indeed, the quadrilateral $BH_CCH_B$ is a trapezoid, and so the line joining the midpoints $M_1$ and $M_2$ of the segments $BC$ and $H_CH_B$ is parallel to $BH_C$ and $CH_B$, which are in turn perpendicular to $AP$. But the segment $GG_A$ is obtained by scaling $M_1M_2$ about point $A$ by a factor of $\frac23$. So, $GG_A \perp AP$. Similarly, one shows $GG_B$ and $GG_C$ are perpendicular to $BP$ and $CP$. It is now clear that $\ma G_CGG_B = \ma CPB$.\\
		
		Let $D_1$ and $E$ denote the midpoints of segments $AH_B$ and $BH_A$. Then, $G_AG_B$ is obtained by scaling $D_1E$ about $H_C$ by a factor of $\frac23$. Hence, $G_AG_B \parallel D_1E$. Similarly, if we let $D_2$ and $F$ denote the midpoints of $AH_C$ and $CH_A$, then $G_AG_C \parallel D_2F$. However, $D_1E$ and $D_2F$ are lines joining the midpoints of parallel chords of a conic ($\HH_P$). Thus, the point $Z$ is common to lines $D_1E$ and $D_2F$. So, we have an equality of directed angles (modulo $\pi$), $\ma G_CG_AG_B = \ma (D_2F, D_1E) = \ma FZE$.\\
		
		Now, let $D_0$ be the midpoint of $PH_A$. But $Z$ lies on the nine point circle of $BPC$, which is the circumcircle of $EFD_0$. Hence, we have an equality of angles 
		\begin{equation*}
			\begin{split}
				\ma G_CG_AG_B & = \ma FZE \\
				& = \ma FD_0E \\
				& = \ma CPB \,\text{(scaling at $H_A$ by a factor of $2$)}\\
				& = \ma G_CGG_B,
			\end{split}
		\end{equation*}
		after which the theorem is immediate.
	\end{proof}
	
	\section{A Special Case}
	Let $I$ be the incenter of $ABC$. The orthocenters of triangles $BIC,$ $CIA,$ and $AIB$ are well-studied. We shall study the orthocenters in detail in the coming sections. We make use of the notation from the previous section, and we replace $P,$ $P',$ and $\HH_P$ by $I,$ $I',$ and $\HH_I$.\\
	
	The Nagel point $N_a$ is very relevant in this setting (see \cite{vonk}). In fact, we have the following:
	\begin{prop}
		\label{prop:nagel}
		The common orthocenter of $\tri_A, \tri_B,$ and $\tri_C$ is $N_a$.
	\end{prop}
	\begin{figure}[h]
		\includegraphics[scale=0.8]{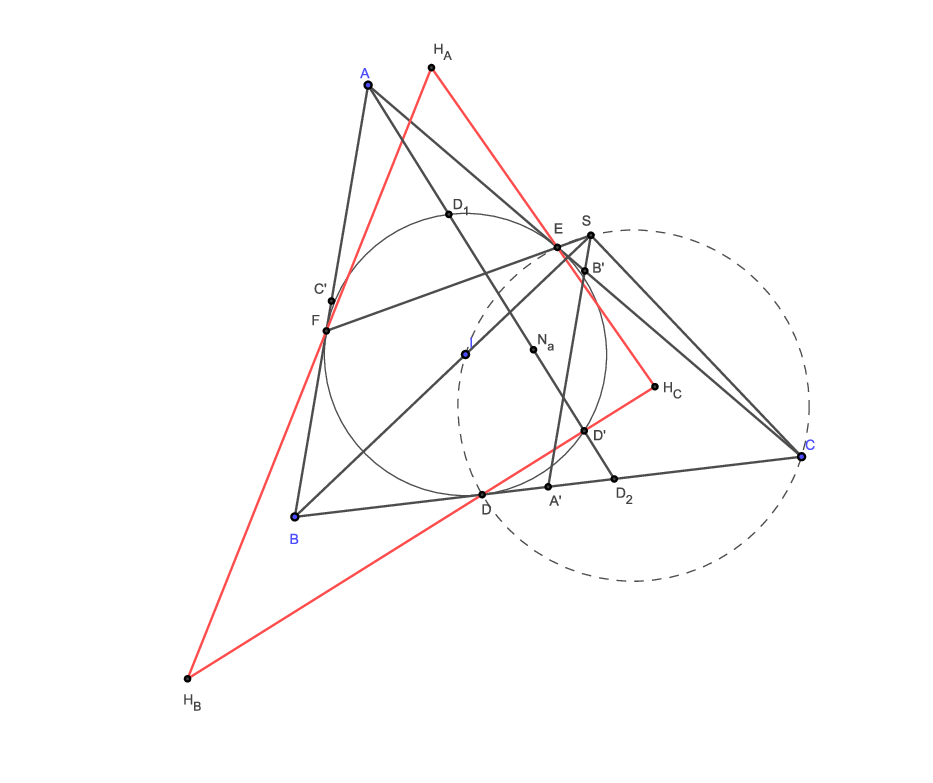}
		\caption{Nagel Point is the common Orthocenter}
	\end{figure}
	In view of the proof of Proposition~\ref{prop:orthologic}, it suffices to show that $AN_a$ is perpendicular to $H_BH_C$. We require a well-known Lemma to show this.
	\begin{lemma}
		\label{lemma:polar_incircle}
		The polar of the midpoint of $BC$ with respect to the incircle of $ABC$ is $H_BH_C$. 
	\end{lemma}
	\begin{proof}
		Let $A',B',$ and $C'$ denote the midpoints of $BC,CA,$ and $AB$. Let $D,E,$ and $F$ denote the points where the incircle touches $BC,CA,$ and $AB$. It suffices to show that $B'C'$ is the polar of $H_A$, because of La Hire's Theorem.\\
		
		Let $BI$ meet $B'C'$ at $T$ and $A'B'$ at $S$. It is well-known that $T$ lies on $DE$ (see \cite{zhao}, \textbf{Lemma 8}). $B'T \parallel BC$ and $CE=CD$ imply that $B'E=B'T$. Similarly $B'E=B'S$ and hence $B'$ is the circumcenter of $ETS$. Since $IE \perp B'E$, the incircle and the circumcircle of $ETS$ are orthogonal. Thus, $T$ and $S$ are inverses of each other in the incircle. It is also well-known that $S,I,E,D,C$ are concyclic (see \cite{zhao}, \textbf{Lemma 8}), and hence $IS \perp CS$. \\
		
		This means that $CS$ is the polar of $T$ in the incircle, and $H_A$ lies on $CS$. By La Hire's theorem, polar of $H_A$ passes through $T=BI \cap B'C'$. A similar argument shows that the polar also passes through $CI \cap B'C'$. Therefore the polar is $B'C'$.
	\end{proof}
	
	\begin{proof} (Of Proposition \ref{prop:nagel})
		Let the other tangent from $A'$ intersect the incircle at $D'$. By Lemma~\ref{lemma:basic2}, the points $D,D'$ lie on $H_BH_C$. Let $AN_a$ intersect the incircle first at $D_1$ and let $D_2$ be the intersection of $AN_a$ with $BC$. It is well-known that $D_1$ is the antipode of $D$, and that $D_2$ and $D$ are isotomic in $BC$ (see \cite{zhao}, \textbf{Lemma 2}). Hence, by midpoint theorem, $IA' \parallel D_1D_2$. But, $IA' \perp DD'$ since $DD'$ is the polar of $A'$. Hence we obtain the perpendicularity $$AN_a \perp DD'=H_BH_C.$$ 
	\end{proof}
	
	\subsection{The Fuhrmann Triangle}
	We present a few results about the Fuhrmann triangle (see \cite{wolfram}). Let $\triangle_a$, $\triangle_b$, and $\triangle_c$ denote the triangles $H_ABC$,$H_BCA$ and $H_CBA$ respectively. 
	\begin{definition}
		\label{def:fuhrmann}
		The triangle formed by circumcenters $O_a,O_b,$ and $O_c$ of $\triangle_a$, $\triangle_b$, and $\triangle_c$ is called the Fuhrmann Triangle $\triangle_f$ of $ABC$. The circumcircle $\Omega_f$ of $\triangle_f$ is called the Fuhrmann Circle of $ABC$.
	\end{definition}
	Let $\Omega$ be the circumcircle of $ABC$, and let $AI, BI, CI$ meet $\Omega$ again at $A_0, B_0, C_0$ respectively. It is easy to see that $A_0$ is the circumcenter of $BIC$. Hence, $O_a$ is the reflection of $A_0$ in $BC$. Similarly, $O_b$ and $O_c$ are reflections of $B_0$ and $C_0$ in $AC$ and $AB$ respectively. Hence, $\Omega_f$ is the Hagge circle (\cite{telv}) of $I$ with respect to $ABC$. All of the following four propositions follow from the properties of the Hagge circle proved in \cite{telv}. However, we shall provide alternative proofs for some of them (due to AoPS user \textit{enhanced}).
	
	\begin{prop}
		\label{prop:ortho_on_fuhrmann}
		The orthocenter $H$ lies on $\Omega_f$.
	\end{prop}
	\begin{proof}
		We shall not retain the names of points introduced in this proof. Let $D, E,$ and $F$ denote the feet of the altitudes from $A,B,$ and $C$ in $ABC$. Let $\Phi$ denote the negative inversion centered at $H$ swapping the pairs $(A,D), (B,E),$ and $(C,F)$. \\
		
		The circumcircle of $BHC$ is the reflection of $\Omega$ about $BC$. Since the midpoint of arc $BC$ reflects about to $O_a$, it follows that $HO_a$ is bisects $\angle BHC$. Under $\Phi$, the circumcircle of $BHC$ maps to line $EF$. Consequently, $\Phi(O_a)$ is the intersection of $HO_a$ with $EF$. \\
		
		From the angle bisector theorem, we have $$\frac{\overline {E\Phi(O_a)}}{\overline {F\Phi(O_a)}} = -\frac{EH}{FH}.$$ Similarly, one gets $$\frac{\overline {F\Phi(O_b)}}{\overline {D\Phi(O_b)}} = -\frac{FH}{DH},$$ and $$\frac{\overline {D\Phi(O_c)}}{\overline {E\Phi(O_c)}} = -\frac{DH}{EF}.$$ From Menelaus' theorem, we conclude that the images of $O_a, O_b, O_c$ under $\Phi$ are collinear. Therefore, $H$ lies on the circumcircle $\Omega_f$ of $O_aO_bO_c$.
	\end{proof}
	
	\begin{prop}
		\label{prop:fuhrmann_center}
		The center $O_f$ of $\Omega_f$ is the reflection of $I$ in the nine point center $N$.
	\end{prop}
	\begin{proof}
		Once again, we shall discard the object names introduced in the proof. Let $A^*$ denote the antipode of $A$ in $\Omega$. The segments $HA^*, O_aA_0,$ and $BC$ have the same midpoint, so in particular, $HO_aA^*A_0$ is a parallelogram. It follows from $AA_0 \perp A_0A^*$ that $AI \perp HO_a$. \\
		
		The reflections of $H, O_a$ about $BC$ lie on $\Omega$. Hence, the perpendicular bisector $\ell_1$ of $HO_a$ passes through the reflection $O_1$ of $O$ about $BC$. But, $AHO_1O$ is a parallelogram, so $N$ is the midpoint of $AO_1$. Both $AI$ and $\ell_1$ are perpendicular to $HO_a$, so the reflection of $I$ about $N$ lies on $\ell_1$. Similarly, the reflection lies on the perpendicular bisectors of $HO_b$ and $HO_c$, from which it is immediate that this reflection must be $O_f$.
	\end{proof}

	\begin{prop}
		\label{prop:antipode}
		$N_a$ is the antipode of $H$ in $\Omega_f$.
	\end{prop}
	\begin{proof}
		From the well-known property of the Nagel line, we have $$\frac{\overline{IG}}{\overline{GN_a}} = \frac{\overline{OG}}{\overline{GH}} = \frac12.$$ From Proposition \ref{prop:fuhrmann_center}, we have the relation $$\frac{\overline{IN}}{\overline{NO_f}} = \frac{\overline{ON}}{\overline{NH}} = 1.$$ Consequently, the lines $HN_a$ and $HO_f$ are parallel to $OI$, which implies that $H, N_a,$ and $O_f$ are collinear. \\
		
		Applying Menelaus' theorem to $IN_aO_f$ with the transversal $GNH$, we find that $$\frac{\overline{O_fH}}{\overline{HN_a}} \cdot  \frac{\overline{IG}}{\overline{GN_a}} \cdot \frac{\overline{IN}}{\overline{NO_f}} = -1.$$ It follows that $\overline{HN_a} = -2 \overline{O_fH}$, so that $N_a$ and $H$ are antipodes in $\Omega_f$.
	\end{proof}
	\begin{theorem}[Stevanovik]
		The orthocenter of $\triangle_f$ is $I$.
	\end{theorem}
	\begin{proof}
		This follows from Property 3 of \cite{telv}, after noticing that $I$ is the orthocenter of $A_0B_0C_0$.
	\end{proof}

	\begin{prop}
		The Euler reflection point of $\tri_f$ is $H$.
	\end{prop}
	\begin{proof}
		This is a direct consequence of Corollary 4.1 of \cite{telv}.  
	\end{proof}

	\section{The Anti-Fuhrmann Triangle}
	In this section, we introduce a new triangle, which bears very similar structure to the Fuhrmann triangle.
	\begin{definition}
		We call the triangle formed by circumcenters of $\triangle_A$, $\triangle_B$, and $\triangle_C$ the Anti-Fuhrmann Triangle $\triangle_r$ of $ABC$. The circumcircle $\Omega_r$ of $\triangle_r$ is called the Anti-Fuhrmann Circle of $ABC$.
	\end{definition}
	\begin{figure}
		\includegraphics[scale=0.8]{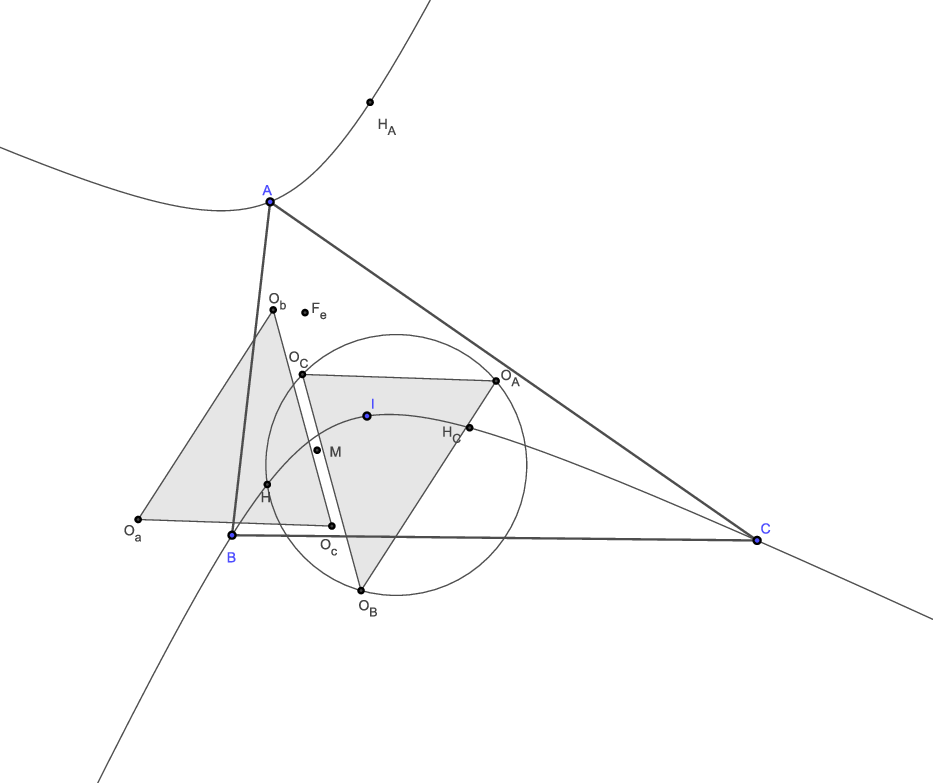}
		\caption{Fuhrmann and Anti-Fuhrmann Triangle}
	\end{figure}
	\begin{theorem}
		\label{thm:anti_fuhrmann}
		The Anti-Fuhrmann triangle $\tri_r$ is the reflection of the Fuhrmann triangle $\tri_f$ about the midpoint of $HI$.
	\end{theorem}
	
	The proof uses the following simple lemma:
	
	\begin{lemma}
		\label{lemma:centroid_ref}
		Let $G_A,G_B,G_C,G,G_a,G_b,G_c,G'$ be the centroids of $\triangle_A$, $\triangle_B$, $\triangle_C$, $ABC$, $\triangle_a$, $\triangle_b$, $\triangle_c$, $H_AH_BH_C$ respectively. Then $GG_AG_BG_C$ and $G'G_aG_bG_c$ are reflections about a point. 
	\end{lemma}
	\begin{proof}
		Let $T$ denote the centroid of the six points $A,B,C,H_A,H_B,H_C$. Then $T$ is the midpoint of segments $GG', G_AG_a, G_BG_b, G_CG_c$. Therefore $GG_AG_BG_C$ and $G'G_aG_bG_c$ are reflections of each other in $T$.
	\end{proof}
	\\
	
	\begin{proof}
		(Of Theorem ~\ref{thm:anti_fuhrmann}) Note that $I$ is the common orthocenter of $\triangle_a,\triangle_b$ and $\triangle_c$ . It follows from taking a homothety at $I$ with ratio $\frac{3}{2}$ that $\overrightarrow{O_aO_b}=\frac{3}{2}\cdot \overrightarrow{G_aG_b}$. Similarly, since $\triangle_A, \triangle_B,\triangle_C$ share an orthocenter (Proposition ~\ref{prop:nagel}), it follows that $\overrightarrow{O_AO_B}=\frac{3}{2}\cdot \overrightarrow{G_AG_B}$. From Lemma ~\ref{lemma:centroid_ref}, $\overrightarrow{G_AG_B}=-\overrightarrow{G_aG_b}$. Thus, $\overrightarrow{O_AO_B}=-\overrightarrow{O_aO_b}$. Analogous relations imply $\triangle_r$ is a reflection of $\triangle_f$ in some point. \\\\
		We proceed to show that this point is actually the midpoint of $HI$. Let $H^*$ be a point such that $IH^*$, $O_AO_a, O_BO_b, O_CO_c$ share a midpoint. Then it follows that $IO_A\parallel H^*O_a$. From Theorem \ref{thm:concyclic}, $G$ lies on the circumcircle of $G_AG_BG_C$. It is a well-known property of the Nagel line that $\frac{\overline{IG}}{\overline{GN_a}}=\frac12$.
		Taking a homothety at $N_a$ with ratio $\frac{3}{2}$ shows that $I$ lies on $\Omega_r$. Hence, $H^*$ lies on $\Omega_f$. The claim in the proof of Theorem~\ref{thm:concyclic} implies that $GG_A \perp AI$. The homothety at $N_a$ with ratio $\frac{3}{2}$ also implies that $O_AI\perp AI$. Hence $H^*O_a \perp AI$. Therefore $H^*$ is a point on $\Omega_f$, such that $H^*O_a\perp AI$. In view of the proof of Proposition ~\ref{prop:fuhrmann_center}, $HO_a\perp AI$. From Proposition ~\ref{prop:ortho_on_fuhrmann}, $H$ lies on $\Omega_f$. Thus $H^*=H$. It immediately follows that $\triangle_r$ and $\triangle_f$ are reflections of each other in the midpoint of $HI$.
	\end{proof}
	
	\begin{corollary}
		\label{cor:anti-fuhrmann}
		The points $H$ and $I$  are the orthocenter and Euler reflection point of $\tri_r$ respectively.
	\end{corollary}
	
	\subsection{Further Properties}
	\begin{prop}
		The center of $\Omega_r$ is the reflection of $O$ about $I$. 
	\end{prop}
	\begin{proof}
		Let $O_r$ denote the center of $\Omega_r$ and let $M$ denote the midpoint of $HI$. It follows from Theorem~\ref{thm:anti_fuhrmann} that $M$ is also the midpoint of $O_rO_f$. Then, $O_fIO_rH$ is a parallelogram. Also, as a consequence of Proposition~\ref{prop:fuhrmann_center}, $OIHO_f$ is a parallelogram. It follows that $O_r$ is the reflection of $O$ about $I$.
	\end{proof}
	
	Recall that the circumcircles of $\triangle_A, \tri_B,$ and $\tri_C$ are concurrent at a point because of Corollary~\ref{cor:conc}. 
	
	\begin{prop}
		\label{prop:common_ortho_2}
		The common point $X$ of the circumcircles of $\triangle_A, \tri_B,$ and $\tri_C$ lies on the line $O_rH$.
	\end{prop}
	\begin{proof}
		Since $N$ is the midpoint of $OH$ and $I$ is the midpoint of $OO_r$, we have $$IN \parallel O_rH.$$ The point $O_f$ lies on $IN$ from Proposition~\ref{prop:fuhrmann_center}, and so does $F_e$ because of Feuerbach's theorem. Using Lemma~\ref{lemma:basic2} and Proposition~\ref{prop:nagel}, we find that $F_e$ is the midpoint of $O_rX$. Moreover, $O_f$ is the midpoint of $N_aH$ (Proposition~\ref{prop:antipode}), so we have $$HX \parallel O_fF_e = IN.$$ It follows that $X$ lies on $O_rH$.  
	\end{proof}
	
	Finally, we go a little bit further than Proposition~\ref{cor:anti-fuhrmann}.
	
	\begin{prop}
		The reflections of the Euler line of $\tri_r$ about its sidelines are precisely the lines $H_AI, H_BI,$ and $H_CI$.
	\end{prop}	
	\begin{proof}
		Using Proposition~\ref{prop:common_ortho_2} and Corollary~\ref{cor:anti-fuhrmann}, we find that $X$ lies on the Euler line of $\tri_r$. Since $I$ is the Euler reflection point of $\tri_r$ (again Corollary~\ref{cor:anti-fuhrmann}), we find that the reflection $XH$ about $O_BO_C$ passes through $I$. Since $X$ and $H_A$ are common to the circumcircles of $\tri_A$ and $\tri_B$, they are also reflections about $O_BO_C$. Consequently, the reflection of $XH$ about $O_BO_C$ is $IH_A$, after which the proof is immediate.  
	\end{proof}
	
	\bibliographystyle{siam}
	\bibliography{references}

	\bigskip
	
	\bigskip
	
	\bigskip
	
	DEPARTMENT OF MATHEMATICS
	
	INDIAN INSTITUTE OF SCIENCE
	
	BANGALORE, INDIA
	
	\textit{E-mail address}: \texttt{sudharshankv02@gmail.com}
	
	\bigskip
	
	\bigskip
	
	DEPARTMENT OF MATHEMATICS
	
	INDIAN INSTITUTE OF TECHNOLOGY
	
	BOMBAY, INDIA
	
	\textit{E-mail address}: \texttt{shantanu.h.nene@gmail.com}\bigskip
	
	\newpage
	
\end{document}